\documentclass[a4paper,11pt]{amsart}

\usepackage[left=1.5in,right=1.5in,top=1.4in,bottom=1.4in]{geometry}
\usepackage[OT4]{fontenc}
\usepackage{amsthm,amsmath,amssymb}
\usepackage{cmtt}
\usepackage{enumitem}
\usepackage{graphicx}
\usepackage[pdftitle={Triangle-free intersection graphs of line segments with large chromatic number},
            pdfauthor={J. Kozik, T. Krawczyk, M. Lason, P. Micek, A. Pawlik, W. T. Trotter, B. Walczak},
            colorlinks=true,linkcolor=black,citecolor=black,filecolor=black,urlcolor=black]{hyperref}

\newtheorem{theorem}{Theorem}
\newtheorem{lemma}[theorem]{Lemma}
\theoremstyle{definition}
\newtheorem{problem}{Problem}
\newtheorem*{update}{Update}

\newcommand{\cgF}{\mathcal{F}}
\newcommand{\cgP}{\mathcal{P}}
\newcommand{\cgQ}{\mathcal{Q}}
\newcommand{\cgS}{\mathcal{S}}

\let\leq\leqslant
\let\geq\geqslant

\def\NN{\mathbb{N}}
\def\RR{\mathbb{R}}

\renewenvironment{enumerate}{\begin{enumorig}[label=\textup{(\roman*)}, noitemsep, topsep=1.5mm plus 1.5mm, leftmargin=*, widest=iii]}{\end{enumorig}}

\makeatletter
\let\old@setaddresses\@setaddresses
\def\@setaddresses{\bigskip\bgroup\parindent 0pt\let\scshape\relax\old@setaddresses\egroup}
\makeatother

\title[Triangle-free segment graphs with large chromatic number]{Triangle-free intersection graphs of line segments with large chromatic number}

\author[Pawlik\and Kozik\and Krawczyk\and Laso\'n\and Micek\and Trotter\and Walczak]{Arkadiusz Pawlik\and Jakub Kozik\and Tomasz Krawczyk\and Micha\l{} Laso\'n\and Piotr Micek\and William T. Trotter\and Bartosz Walczak}

\thanks{A journal version of this paper appeared in \emph{J.\ Combin.\ Theory Ser.\ B}, 105:6--10, 2014.}
\thanks{Five authors were supported by Ministry of Science and Higher Education of Poland grant 884/N-ESF-EuroGIGA/10/2011/0 within ESF EuroGIGA project GraDR\@.}

\address[Arkadiusz Pawlik, Jakub Kozik, Tomasz Krawczyk, Piotr Micek, Bartosz Walczak]{Theoretical Computer Science Department, Faculty of Mathematics and Computer Science, Jagiellonian University, Krak\'ow, Poland}
\email{\mtt\{pawlik,jkozik,krawczyk,micek,walczak\mtt\}@tcs.uj.edu.pl}

\address[Micha\l{} Laso\'n]{Theoretical Computer Science Department, Faculty of Mathematics and Computer Science, Jagiellonian University, Krak\'ow, Poland; Institute of Mathematics of the Polish Academy of Sciences, Warsaw, Poland}
\email{michalason@gmail.com}

\address[William T. Trotter]{School of Mathematics, Georgia Institute of Technology, Atlanta, GA 30332, USA}
\email{trotter@math.gatech.edu}

\begin{document}

\begin{abstract}
In the 1970s, Erd\H{o}s asked whether the chromatic number of intersection graphs of line segments in the plane is bounded by a function of their clique number.
We show the answer is no.
Specifically, for each positive integer $k$, we construct a triangle-free family of line segments in the plane with chromatic number greater than $k$.
Our construction disproves a conjecture of Scott that graphs excluding induced subdivisions of any fixed graph have chromatic number bounded by a function of their clique number.
\end{abstract}

\maketitle

\section{Introduction}

A \emph{proper coloring} of a graph is an assignment of colors to the vertices of the graph such that no two adjacent ones are assigned the same color.
The minimum number of colors sufficient to color a graph $G$ properly is called the \emph{chromatic number} of $G$ and denoted by $\chi(G)$.
The maximum size of a clique (a set of pairwise adjacent vertices) in a graph $G$ is called the \emph{clique number} of $G$ and denoted by $\omega(G)$.
It is clear that $\chi(G)\geq\omega(G)$.
A class of graphs is \emph{$\chi$-bounded} if there is a function $f\colon\NN\to\NN$ such that $\chi(G)\leq f(\omega(G))$ holds for any graph $G$ from the class\footnotemark.
\footnotetext{This notion has been introduced by Gy\'arf\'as \cite{Gya87}, who called the class \emph{$\chi$-bound} and the function \emph{$\chi$-binding}.
However, the term \emph{$\chi$-bounded} seems to be better established in the modern terminology.}
A \emph{triangle} is a clique of size $3$.
A graph is \emph{triangle-free} if it does not contain any triangle.

There are various constructions of triangle-free graphs with arbitrarily large chromatic number.
The first one was given by Zykov \cite{Zyk49}, and the one perhaps best known is due to Mycielski \cite{Myc55}.
On the other hand, in the widely studied class of perfect graphs, which include interval graphs, split graphs, chordal graphs, comparability graphs, etc., the chromatic number and the clique number are equal.
In particular, these graphs are $2$-colorable when triangle-free.

The \emph{intersection graph} of a family of sets $\cgF$ is the graph with vertex set $\cgF$ and edge set consisting of pairs of intersecting elements of $\cgF$.
For simplicity, we identify the family $\cgF$ with its intersection graph.

The study of the relation between $\chi$ and $\omega$ for geometric intersection graphs was initiated by Asplund and Gr\"unbaum \cite{AG60}.
They proved that the families of axis-aligned rectangles in the plane are $\chi$-bounded.
On the other hand, Burling \cite{Bur65} showed that the triangle-free families of axis-aligned boxes in $\RR^3$ have arbitrarily large chromatic number.

Paul Erd\H{o}s asked in the 1970s\footnote{An approximate date confirmed in personal communication with Andr\'as Gy\'arf\'as and J\'anos Pach; see also \cite[Problem 1.9]{Gya87} and \cite[Problem 2 in Section 9.6]{BMP-book}.} whether the families of line segments in the plane are $\chi$-bounded.
Kratochv\'{\i}l and Ne\v{s}et\v{r}il generalized this question to the families of curves in the plane any two of which intersect at most once, see \cite{KN95}.
McGuinness \cite{McG00} conjectured that the families $\cgF$ of bounded arcwise connected sets in the plane such that for any $S,T\in\cgF$, the set $S\cap T$ is arcwise connected or empty are $\chi$-bounded.

Some important special cases of Erd\H{o}s's question are known to have positive solutions.
In particular, Suk \cite{Suk14} proved that the families of segments intersecting a common line and the families of unit-length segments are $\chi$-bounded, improving results of McGuinness for triangle-free families \cite{McG00}.

We show that the answer to Erd\H{o}s's question is negative.
Namely, for every positive integer $k$, we construct a triangle-free family $\cgS$ of line segments in the plane such that $\chi(\cgS)>k$.

A related question of Erd\H{o}s\footnote{See \cite[Problem 1.10]{Gya87}.} asks whether the class of complements of intersection graphs of line segments in the plane is $\chi$-bounded.
Here the answer is positive as shown by Pach and T\"or\H{o}csik \cite{PT94}.

As first observed by Fox and Pach \cite{FP-priv}, our result disproves a purely graph-theoretical conjecture of Scott \cite{Sco97} that for every graph $H$, the class of graphs excluding induced subdivisions of $H$ is $\chi$-bounded.
To see this, take $H$ being the $1$-subdivision of a non-planar graph, and note that no subdivision of such $H$ is representable as an intersection graph of segments.

\section{Proof}\label{sec:segments}

\begin{theorem}\label{thm:main}
For every integer\/ $k\geq 1$, there is a family\/ $\cgS$ of line segments in the plane with no three pairwise intersecting segments and with\/ $\chi(\cgS)>k$.
\end{theorem}

We actually prove a stronger and more technical lemma, which admits a relatively compact inductive proof.

Let $\cgS$ be a family of line segments contained in the interior of a rectangle $R=[a,c]\times[b,d]$.
A rectangle $P=[a',c]\times[b',d']$ is a \emph{probe} for $(\cgS,R)$ if the following conditions are satisfied:
\begin{enumerate}
\item We have $a<a'<c$ and $b<b'<d'<d$.
Note that the right boundary of $P$ lies on the right boundary of $R$.
\item No line segment in $\cgS$ intersects the left boundary of $P$.
\item No line segment in $\cgS$ has an endpoint inside or on the boundary of $P$.
\item The line segments in $\cgS$ intersecting $P$ are pairwise disjoint.
\end{enumerate}
The reader should envision a probe as a thin rectangle entering $R$ from the right and intersecting an independent set of line segments.
The purpose of the restriction on the left boundary of the probe is to simplify the details of the construction to follow.
The rectangle $[a',c']\times[b',d']$ with maximum $c'$ that is internally disjoint from every line segment in $\cgS$ is the \emph{root} of $P$.

In the argument below, we construct a family of pairwise disjoint probes for $(\cgS,R)$.
We illustrate such a configuration in Figure~\ref{fig:probes}.

\begin{figure}[t]
\begin{center}
\includegraphics[scale=.53]{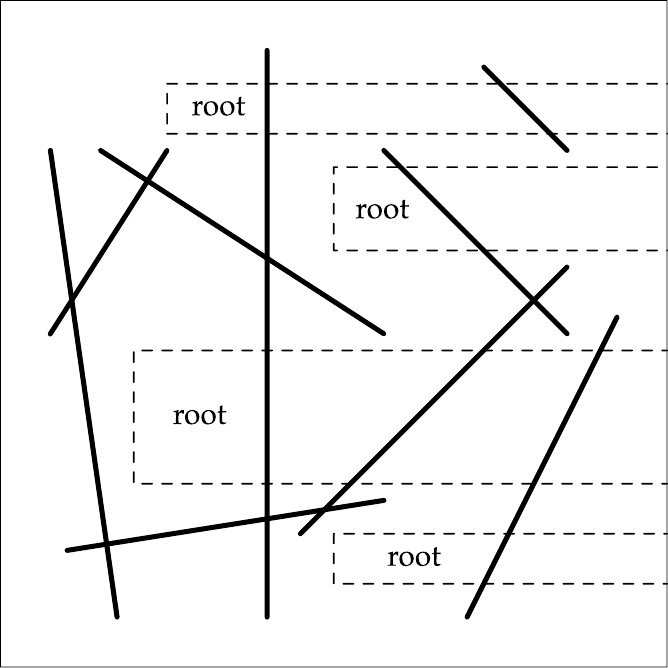}
\end{center}
\caption{Segments, probes and roots}
\label{fig:probes}
\end{figure}

We define sequences $(s_i)_{i\in\NN}$ and $(p_i)_{i\in\NN}$ by induction, setting $s_1=p_1=1$, $s_{i+1}=(p_i+1)s_i+p_i^2$, and $p_{i+1}=2p_i^2$.

\begin{lemma}\label{lem:induction}
Let\/ $k\geq 1$ and\/ $R$ be an axis-aligned rectangle with positive area.
There is a triangle-free family\/ $\cgS_k$ of\/ $s_k$ line segments in the interior of\/ $R$ and a family\/ $\cgP_k$ of\/ $p_k$ pairwise disjoint probes for\/ $(\cgS_k,R)$ such that for any proper coloring\/ $\phi$ of\/ $\cgS_k$, there is a probe\/ $P\in\cgP_k$ for which\/ $\phi$ uses at least\/ $k$ colors on the segments in\/ $\cgS_k$ intersecting\/ $P$.
\end{lemma}

\begin{proof}
The proof goes by induction on $k$.
For the base case $k=1$, we pick any non-horizontal segment inside $R$ as the only member of $\cgS_1$.
The probe in $\cgP_1$ is any rectangle contained in $R$, touching the right boundary of $R$ and piercing the chosen segment with its lower and upper boundaries.

Now goes the induction step: for a given rectangle $R$, we construct $\cgS_{k+1}$ and $\cgP_{k+1}$.
First, we draw a family $\cgS=\cgS_k$ inside $R$ and let $\cgP=\cgP_k$ be the associated set of probes, as claimed by the induction hypothesis.
Then, for each probe $P\in\cgP$, we place another copy $\cgS_P$ of $\cgS_k$ with set of probes $\cgQ_P$ inside the root of $P$.
Finally, for every $P\in\cgP$ and every $Q\in\cgQ_P$, we draw the \emph{diagonal} of $Q$, that is, the line segment $D_Q$ from the bottom-left corner of $Q$ to the top-right corner of $Q$.
Note that $D_Q$ crosses all segments pierced by $Q$ and no other segment.
The family $\cgS_{k+1}$ consists of the line segments from the $p_k+1$ copies of $\cgS_k$ and the $p_k^2$ diagonals, so the total number of segments is $(p_k+1)s_k+p_k^2=s_{k+1}$.

Now we show how the set of probes $\cgP_{k+1}$ is constructed.
For $P\in\cgP$ and $Q\in\cgQ_P$, let $\cgS(P)$ be the segments in $\cgS$ that intersect $P$ and $\cgS_P(Q)$ be the segments in $\cgS_P$ that intersect $Q$.
For each $P\in\cgP$ and each $Q\in\cgQ_P$, we put two probes into $\cgP_{k+1}$:\ a \emph{lower probe} $L_Q$ and an \emph{upper probe} $U_Q$, both lying inside $Q$ extended to the right boundary of $R$.
We choose the lower probe $L_Q$ very close to the bottom edge of $Q$ and thin enough so that it intersects all segments in $\cgS_P(Q)$ but not the diagonal $D_Q$.
We choose the upper probe $U_Q$ very close to the top edge of $Q$ and thin enough so that it intersects the diagonal $D_Q$ but not the segments in $\cgS_P(Q)$.
Both $L_Q$ and $U_Q$ end at the right boundary of $R$, as required by the definition of probe, and thus intersect also the segments in $\cgS(P)$.
Note that $L_Q$ and $U_Q$ are disjoint.
By the induction hypothesis and the placement of $\cgS_P$ inside the root of $P$, the sets $\cgS(P)\cup\cgS_P(Q)$ and $\cgS(P)\cup\{D_Q\}$ are both independent, so $L_Q$ and $U_Q$ indeed satisfy the conditions for being probes.
See Figure~\ref{fig:diagonal} for an illustration.
Clearly, the probes in $\cgP_{k+1}$ are pairwise disjoint and their total number is $2p_k^2=p_{k+1}$.

\begin{figure}[t]
\begin{center}
\includegraphics[scale=.53]{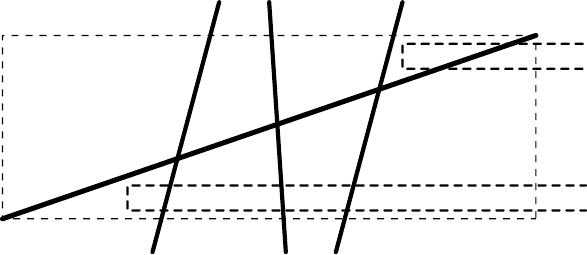}
\end{center}
\caption{A diagonal with lower and upper probes}
\label{fig:diagonal}
\end{figure}

The family $\cgS_{k+1}$ is triangle-free, because we constructed $\cgS_{k+1}$ by taking disjoint copies of triangle-free families and adding diagonals intersecting independent sets of segments.
Let $\phi$ be a proper coloring of $\cgS_{k+1}$.
We show that there is a probe in $\cgP_{k+1}$ for which $\phi$ uses at least $k+1$ colors on the line segments in $\cgS_{k+1}$ intersecting that probe.
Consider the restriction of $\phi$ to $\cgS$, the original copy of $\cgS_k$ used to launch the construction.
There is a probe $P\in\cgP$ such that $\phi$ uses at least $k$ colors on the line segments in $\cgS$ intersecting $P$.
Now, consider $\cgS_P$, the copy of $\cgS_k$ put inside the root of $P$.
Again, there is a probe $Q\in\cgQ_P$ such that $\phi$ uses at least $k$ colors on the segments in $\cgS_P$ intersecting $Q$.
If $\phi$ uses different sets of colors on the segments intersecting $P$ and the segments intersecting $Q$, then at least $k+1$ colors are used on the segments pierced by the lower probe $L_Q$.
If $\phi$ uses the same set of colors on the segments intersecting $P$ and those intersecting $Q$, then another color must be used on the diagonal $D_Q$, and thus $\phi$ uses at least $k+1$ colors on the segments intersecting the upper probe $U_Q$.
\end{proof}

The smallest family $\tilde\cgS_k$ of segments satisfying the conclusion of Theorem \ref{thm:main} that we know is obtained by taking all segments from the family $\cgS_k$ constructed above and adding the diagonals of all probes in $\cgP_k$.
It is indeed triangle-free as the segments in $\cgS_k$ intersecting every probe form an independent set.
Since every proper coloring of $\cgS_k$ uses at least $k$ colors on the segments intersecting some probe, the diagonal of this probe must receive yet another color, which yields $\chi(\tilde\cgS_k)>k$.
By an argument similar to that in the proof of Lemma \ref{lem:induction}, we can show that the family $\tilde\cgS_k$ is $(k+1)$-critical, which means that $\chi(\tilde\cgS_k)=k+1$ and removing any segment from $\tilde\cgS_k$ decreases the chromatic number to $k$.

It should be noted that the family $\tilde\cgS_k$ and the triangle-free family of axis-aligned boxes in $\RR^3$ with chromatic number greater than $k$ constructed by Burling \cite{Bur65} yield the same intersection graph.

\section{Remarks}\label{sec:remarks}

An easy generalization of the presented construction shows the following: for every arcwise connected compact set $S$ in the plane that is not an axis-aligned rectangle, the triangle-free families of sets obtained from $S$ by translation and independent scaling in two directions have unbounded chromatic number.
In particular, this gives the negative answer to a question of Gy\'arf\'as and Lehel \cite{GL85} whether the families of axis-aligned L-shapes are $\chi$-bounded.
For some sets $S$ (e.g.\ circles and square boundaries), we can even restrict the transformations to translation and uniform scaling and still obtain graphs with arbitrarily large chromatic number.
We discuss this matter in a follow-up paper \cite{PKK+13}.

\section{Problems}\label{sec:problems}

\begin{problem}
What is (asymptotically) the maximum chromatic number of a triangle-free family of $n$ segments in the plane?
\end{problem}

The size of the triangle-free family $\tilde\cgS_k$ of segments with $\chi(\tilde\cgS_k)>k$ that we construct in Section \ref{sec:segments} is $s_k+p_k$.
It easily follows from the inductive definition of $p_k$ and $s_k$ that
\begin{equation*}
2^{2^{k-1}-1}=p_k\leq s_k\leq 2^{2^{k-1}}-1.
\end{equation*}
Therefore, we have $|\tilde\cgS_k|=\Theta(2^{2^{k-1}})$.
This shows that the maximum chromatic number of a triangle-free family of segments of size $n$ is of order $\Omega(\log\log n)$.
On the other hand, the bound of $O(\log n)$ follows from a result of McGuinness \cite{McG00}.

\begin{problem}
\label{prob:2}
Is there a constant $c>0$ such that every triangle-free family of $n$ segments in the plane contains an independent subfamily of size at least $cn$?
\end{problem}

By the above-mentioned bound $\chi(\cgS)=O(\log n)$, any triangle-free family $\cgS$ of segments of size $n$ contains an independent subfamily of size $\Omega(n/\log n)$.

\begin{update}
Walczak \cite{Wal15} proved the negative answer to the question in Problem~\ref{prob:2}.
\end{update}

The chromatic number of segment intersection graphs containing no triangles and no $4$-cycles is bounded, as shown by Kostochka and Ne\v{s}et\v{r}il \cite{KN98}.
The following problem has been proposed by Jacob Fox.

\begin{problem}
\label{prob:3}
Do the families of segments in the plane containing no triangles and no $5$-cycles have bounded chromatic number?
\end{problem}

\begin{update}
We have learned from Sean McGuinness that the positive answer to the question in Problem \ref{prob:3} is a direct corollary to the following result of him \cite[Theorem 5.3]{McG01}: if a triangle-free segment intersection graph has large chromatic number, then it contains a vertex $v$ with the property that the vertices at distance $2$ from $v$ induce a subgraph with large chromatic number.
Indeed, any two adjacent vertices at distance $2$ from $v$ witness a $5$-cycle.
Iterated application of McGuinness's result shows that for any $\ell\geq 5$, triangle-free segment intersection graphs with chromatic number large enough contain a cycle of length $\ell$.
We are grateful to Sean for these observations.
\end{update}

\section*{Acknowledgments}

We thank Jacob Fox and J\'anos Pach for their helpful remarks and advice.

\bibliographystyle{plain}
\bibliography{segments}

\end{document}